\documentclass[a4paper,11pt]{amsart}

       \usepackage{palatino, verbatim}
        \usepackage[latin1]{inputenc}
        \usepackage[T1]{fontenc}
        \usepackage{amsthm, amsfonts}
        \usepackage{amsfonts}
        \usepackage{graphicx}
        \usepackage{amssymb}
        \usepackage{amsmath}
        \usepackage{latexsym}
        \usepackage[all]{xy}
        \usepackage{color}

        \newtheorem{thm}{Theorem}[section]
          
          \newtheorem{lem}[thm]{Lemma}
          \newtheorem{prop}[thm]{Proposition}
          \newtheorem{conj}[thm]{Conjecture}

        \theoremstyle{definition}
          
          \newtheorem{rem}{Remark}

          \newcommand\M{{M}}

           \newcommand\G{{\mathcal G}}
          \newcommand\im{\mathrm{Im}}
          \newcommand\Pic{\mathrm{Pic}}
          
          \newcommand\oo{\mathcal O}
          \newcommand\Z{\mathbb{Z}}
          \newcommand\Ext{\mathrm{Ext}}
          
          \newcommand\Hom{\mathrm{Hom}}
          
          \newcommand\rk{\mathrm{rk}}

\title
{A codimension 2 component of the Gieseker-Petri locus}
\author{Margherita Lelli--Chiesa}
\address{Università degli studi Roma Tre, Dipartimento di Matematica e Fisica, Largo San Leonardo Murialdo 1, 00146 Roma}
\email{margherita.lellichiesa@uniroma3.it}

\begin{document}
\begin{abstract}
We show that the Brill-Noether locus $M^3_{18,16}$ is an irreducible component of the Gieseker-Petri locus in genus $18$ having codimension $2$ in the moduli space of curves. This result disproves a conjecture predicting that the Gieseker-Petri locus is always divisorial.
\end{abstract}
\maketitle
\section{Introduction}\label{stazione}
The Gieseker-Petri locus $GP_g$ inside the moduli space of smooth irreducible genus $g$ curves $M_g$ parametrizes all those curves $C$ that possess a line bundle $A$ for which the Petri map
$
\mu_{0,A}:H^0(C,A)\otimes H^0(C,\omega_C\otimes A^\vee)\to H^0(C,\omega_C)
$
is non-injective. By the Gieseker-Petri Theorem, $GP_g$ is a proper subvariety of $M_g$ and, by Clifford's Theorem and  the Riemann-Roch Theorem, it breaks up  as follows:
\begin{equation*}GP_g=\bigcup_{0<2r\leq d\leq g-1}GP^r_{g,d},\end{equation*}
where $GP^r_{g,d}$ is its closed subset defined as
$$GP^r_{g,d}:=\{[C]\in\M_g\,\vert\, \exists\,(A,V)\in G^r_d(C)\textrm{ with } \ker\mu_{0,V}\neq 0\};$$
 here, $\mu_{0,V}$ denotes the restriction of $\mu_{0,A}$ to $V\otimes H^0(C,\omega_C\otimes A^\vee)$. Plenty of work has been devoted to the study of the codimension of the Gieseker-Petri locus and this was partially motivated by the following controversial conjecture (cf. \cite{CHF} for a very nice survey of the debate):
\begin{conj}\label{sernesi}
The Gieseker-Petri locus $GP_g$ has pure codimension $1$ in $M_g$.
\end{conj}
The above conjecture is known to hold for low genera thanks to the work of Castorena for $g\leq8$ (cf. \cite{Ca}), and the author herself in the range $9\leq g\leq 13$ (cf. \cite{LC1}). However, in general not very much is known about the dimension of the loci $GP^r_{g,d}$ and their reciprocal position. Note that when the Brill-Noether number $\rho(g,r,d):=g-(r+1)(g-d+r)$ is negative, the Petri map associated with a $g^r_d$ on a genus $g$ curve is authomatically non-injective for dimension reasons and the locus $GP^r_{g,d}$ coincides with the Brill-Noether locus
$$M^r_{g,d}:=\{[C]\in\M_g\,\vert\,W^r_d(C)\neq\emptyset\}.$$
This is an irreducible divisor when $\rho(g,r,d)=-1$ \cite{EH}. However, as soon as $\rho(g,r,d)\leq-2$,  the codimension of $M^r_{g,d}$ in $M_g$ is at least $2$ \cite{St}. Hence, Conjecture \ref{sernesi} would force any Brill-Noether locus $M^r_{g,d}$ with $\rho(g,r,d)\leq -2$ to be contained in some other loci $GP^s_{g,e}$ filling up a divisorial component of $GP_g$. In the present paper we disprove this fact:
\begin{thm}\label{main}
The Brill-Noether locus $M^3_{18,16}$ is an irreducible component of the Gieseker-Petri locus $GP_{18}$ having codimension $2$ in $M_{18}$.
\end{thm}
Now we summarize the results in the literature that concern the loci $GP^r_{g,d}$ with  $\rho(g,r,d)\geq 0$. It was proved by Farkas \cite{F2,F3}  that they always carry a divisorial component. If moreover $\rho(g,r+1,d)<0$, then $GP^r_{g,d}$ has pure codimension $1$ outside $M^{r+1}_{g,d}$ by the work of Bruno and Sernesi \cite{BS}. The problem remains open whether the loci $GP^r_{g,d}$ have pure codimension $1$ in $M_g$ as soon as $\rho(g,r,d)\geq 0$; this guess looks more plausible than Conjecture \ref{sernesi}, even though it is known to hold in very few special cases, namely, when $\rho(g,r,d)=0$ and for the locus $GP^1_{g,g-1}$ parametrizing curves with a vanishing theta-null.

It is worth spending some words on the reason why our counterexample occurs in genus $18$ and not before.  Conjecture \ref{sernesi} up to genus $13$ was proved  by verifying that all the loci $GP^r_{g,d}$ whose codimension is either unknown or strictly larger than $1$ are contained in some divisorial components of $GP_g$; the proof realizes on some general inclusions holding in any genus (that we recall here in Proposition \ref{prop:gp}) along with a few ad hoc arguments. However, similar arguments in genus $14$ fail to control the codimension $2$ Brill-Noether locus $M^3_{14,13}$. This is the first case that highlights the (somehow unexpected) relevance of non-complete linear series in determining the relative position of the loci $GP^r_{g,d}$: it turns out that any genus $14$ curve with a $g^3_{13}$ also possesses a non-complete linear series $g^2_{13}$ with non-injective Petri map. In particular, this implies that $M^3_{14,13}$ is contained in $GP^2_{14,13}$. This phenomenon involving non-complete linear series  occurs any time that $d-g<\rho(g,r,d)<0$ (cf. Proposition \ref{noncomp}). All together, the results in Section \S \ref{due} suggests $18$ to be the lowest genus in which a Brill-Noether locus of codimension $\geq 2$ may provide a counterexample to Conjecture \ref{sernesi} (cf. Remark \ref{genus}). Furthermore, the same results in genus $18$ reduce Theorem \ref{main} to the following:
\begin{thm}\label{explicit}
There exists a smooth irreducible curve $C\subset\mathbb{P}^3$ of degree $16$ and genus $18$ such that all the varieties $G^3_{17}(C)$, $G^2_d(C)$ for $14\leq d\leq 17$ and $G^1_k(C)$ for $10\leq k\leq 17$ are smooth of the expected dimension.
\end{thm}
A curve $C$ as in the above statement is realized in Section \ref{quattro} as a section of a smooth quartic $K3$ surface $S\subset \mathbb P^3$ of Picard number $2$. The Brill-Noether behavior of $C$ is analyzed by means of non-trivial techniques involving higher rank Lazarsfeld-Mukai bundles, that were partially developed in \cite{LC2,LC3}. The definition and some basic properties of Lazarsfeld-Mukai bundles are preliminarly recalled in Section \S \ref{tre}, where they are stated in such a way that they  hold also for non-complete linear series. In fact, the most involving part in the proof of Theorem \ref{explicit} turns out to be the control of the Petri map associated with non-complete linear series of type $g^2_{16}$ (cf. Proposition \ref{noncomplete}): this because the Lazarsfeld-Mukai bundle associated with a non-complete linear series has non-vanishing $h^1$ and thus its automorphism group does not always govern the kernel of the Petri map (even if one chooses $C$ to be general in its linear system).

It is very plausible that one may construct counterexamples to Conjecture \ref{sernesi} for infinitely many genera using curves on $K3$ surfaces. However, when the genus becomes higher, the number of components of the Gieseker-Petri locus $GP_g$ increases and, given a curve $[C]\in M^r_{g,d}$, it will become more and more challenging to exclude that $[C]$ lies in some divisorial component of $GP_g$. In particular, if $C$ is contained in a $K3$ surface $S$, Lazarsfeld-Mukai bundles of very high rank will be involved in the computation. 

\textbf{Acknowledgements:} This problem was suggested to me during my Ph.D. program by my advisor Gavril Farkas: I thank him for many valuable conversations on the topic. I also thank Alessandro D'Andrea for interesting discussions concerning non-complete linear series. The author is a member of PRIN-2017 Project {\em Moduli Theory and Birational Classification} and of GNSAGA.

\section{Components of the Gieseker-Petri locus}\label{due}
In order to determine the irreducible components of $GP_g$, it is necessary to understand the reciprocal position of the loci $GP^r_{g,d}$. We summarize some inclusions holding in any genus (cf. Section 2 in \cite{LC1}):
\begin{prop}\label{prop:gp}
One has that:
\begin{enumerate}
\item[(i)] If $\rho(g,r,d+1)<0$, then $M^r_{g,d}\subset M^r_{g,d+1}$.
\item[(ii)] If $\rho(g,r-1,d-1)<0$ and $r>1$, then $M^r_{g,d}\subset M^{r-1}_{g,d-1}$.
\item[(iii)] If $\rho(g,r,d)\in\{0,1\}$, then $\M^r_{g,d-1}\subset GP^r_{g,d}$ and $\M^{r+1}_{g,d+1}\subset GP^r_{g,d}$.
\item[(iv)] If $d<\lfloor \frac{g+3}{2}\rfloor$, then  $\M^1_{g,d}$ is contained in the Brill-Noether divisor $\M^1_{g,(g+1)/2}$ if $g$ is odd and in the divisor $GP^1_{g,(g+2)/2}$ if $g$ is even. Furthermore, in the latter case any curve in $GP^1_{g,(g+2)/2}$ has a base point free $g^1_{(g+2)/2}$ for which the Petri map is non-injective.
\end{enumerate}
\end{prop}
\begin{proof}
Item (i) is straightforward: if $[C]\in M^r_{g,d}$ and $A\in \Pic^d(C)$ satisfies $h^0(A)\geq r+1$, then for any point $P\in C$ the line bundle $A(P)\in \Pic^d(C)$ satisfies the inequality $h^0(A(P))\geq r+1$, too. 

Item (ii) is \cite[Lem. 2.2]{LC1}, while item (iv) is proved in \cite[Lem. 2.3 and Cor. 2.4]{LC1}.

To obtain item (iii) one can proceed as in the proof of \cite[Lem. 2.5]{LC1} having the foresight to use a general point of $C$ in order to construct a $g^r_d$ from the  $g^r_{d-1}$ or the $g^{r+1}_{d+1}$, in case the latter is not primitive. 
\end{proof}
The above inclusions have been used in \cite{LC1} in order to prove that the Gieseker-Petri locus has pure codimension $1$ in $\M_g$ for $g\leq 13$. However, already in genus $14$ they do not imply the inclusion of the Brill-Noether locus $\M^3_{14,13}$ (which has codimension $2$ in $\M_{14}$) neither in a component of type $GP^r_{14,d}$ with $\rho(14,r,d)\geq 0$ nor in a Brill-Noether divisor. The following result takes care of this component and thus motivates why our counterexample occurs in genus $18$ and not before.
\begin{prop}\label{noncomp}
Let $g,r,d$ be integers such that $\rho(g,r,d)<d-g<0$. Then any genus $g$ curve with a complete $g^r_d$ also possesses a non-complete $g^{r-1}_d$ for which the Petri map is non-injective. In particular, this implies the inclusion $\M^r_{g,d}\setminus \M^{r+1}_{g,d}\subset GP^{r-1}_{g,d}$.
\end{prop}
\begin{proof}
The statement is trivial if $\rho(g,r-1,d)<0$, so we may assume $\rho(g,r-1,d)\geq 0$. We consider a curve $[C]\in\M^r_{g,d}$ possessing a complete linear series $A$ of type $g^r_d$. The kernel of the Petri map $\mu_{0,A}$ has dimension $\geq  -\rho(g,r,d)>g-d$. On the other hand, the space $Z_r$ of tensors in $H^0(C,A)\otimes H^0(C,\omega_C\otimes A^\vee)$ that do not have maximal rank is a Zariski closed subset of codimension equal to $h^0(\omega_C\otimes A^\vee)-r=g-d$; hence, for any linear subspace $X$ of  $H^0(C,A)\otimes H^0(C,\omega_C\otimes A^\vee)$ one has 
$$
\mathrm{codim}_X(X\cap Z_r)\leq g-d,
$$
cf. \cite{Ei}. One obtains the statement setting $X=\ker\mu_{0,A}$.

\end{proof}
\begin{rem}\label{genus}
When $\rho(g,r,d)=-2$, the inequalities $\rho(g,r,d)<d-g<0$ imply $d=g-1$. Furthermore, for $\rho(g,r,d)=-2$ the locus $\M^r_{g,d}$ has codimension $2$ in $\M_g$, while the codimension of $\M^{r+1}_{g,d}$ is strictly larger \cite{St}; hence, no irreducible components of $\M^r_{g,d}$ is contained in $\M^{r+1}_{g,d}$ in this case and Proposition \ref{noncomp} yields the inclusion $M^r_{g,d}\subset GP^{r-1}_{g,d}$. For instance, we obtain that $\M^3_{14,13}\subset GP^2_{14,13}$. In genus $14$ also all the other Brill-Noether loci $\M^r_{14,d}$ with $\rho(14,r,d)<-1$ are included in some loci $GP^s_{14,e}$ with $\rho(14,s,e)\geq 0$ or in some Brill-Noether divisors thanks to Proposition \ref{prop:gp}. 

Analogously, Propositions \ref{prop:gp} and \ref{noncomp} enable us to control all the Brill-Noether loci of codimension $\geq 2$ in genus $g\in\{15,16,17\}$. In particular, in these genera the Gieseker-Petri locus decomposes as
\begin{equation*}GP_g=\bigcup_{\substack{0<2r\leq d\leq g-1\\ \rho(g,r,d)\geq-1}}GP^r_{g,d}.\end{equation*}
However, this does not prove Conjecture \ref{sernesi} up to genus $17$ since it may still fail for some loci $GP^r_{g,d}$ with $\rho(g,r,d)\geq 0$.
\end{rem}

Now we concentrate on the case $g=18$, where Proposition \ref{prop:gp} provides the fol\-lo\-wing decomposition of the Gieseker-Petri locus:
\begin{equation}\label{18}
GP_{18}=GP^3_{18,17}\cup\M^3_{18,16}\cup\bigcup_{d=14}^{17} GP^2_{18,d}\cup\bigcup_{k=10}^{17}GP^1_{18,k};
\end{equation}
here we have used that $\rho(18,2,14)=0$ in order to conclude that $\M^2_{18, 13}\subset GP^2_{18,14}$.
Since $\rho(18,3,16)=-2$, then $\mathrm{codim}_{\M_{18}}\M^3_{18,16}=2$ (cf. \cite{St}). In order to prove that $\M^3_{18,16}$ is an irreducible component of $GP_{18}$, it is enough to verify that
\begin{equation}\label{uffino}
M^3_{18,16}\not\subset GP^3_{18,17}\cup\bigcup_{d=14}^{17} GP^2_{18,d}\cup\bigcup_{k=10}^{17}GP^1_{18,k}.
\end{equation}
Equivalently, one has to provide a curve $C$ as in Theorem \ref{explicit}.

\section{Lazarsfeld-Mukai bundles}\label{tre}
In this section we recall the definition and basic properties of Lazarsfeld-Mukai bundles extending them to non-complete linear series. Let $C$ be a smooth genus $g$ curve lying on a $K3$ surface $S$ and let $(A,V)$ be a base point free $g^r_d$ on $C$.  The Lazarsfeld-Mukai bundle $E_{C,(A,V)}$ is defined as the dual of the kernel of the evaluation map $V\otimes \oo_S\to A$, and thus sits in the following short exact sequence:
\begin{equation}\label{LM}
0\to V^\vee\otimes\oo_S\to E_{C,(A,V)}\to \omega_C\otimes A^\vee\to 0.
\end{equation}
In particular, $E_{C,(A,V)}$ is globally generated off the base locus of $\omega_C\otimes A^\vee$ and both its Chern classes and cohomology are easily computed from \eqref{LM}:
\begin{itemize}
\item $\rk E_{C,(A,V)}=r+1$, $c_1(E_{C,(A,V)})=C$, $c_2(E_{C,(A,V)})=d$;
\item $h^0(E_{C,(A,V)})=r+1+h^1(A)$, $h^1(E_{C,(A,V)})=h^0(A)-r-1$, $h^2(E_{C,(A,V)})=0$.
\end{itemize}
In particular, $h^1(E_{C,(A,V)})=0$ as soon as the linear series $(A,V)$ is complete, that is, $V=H^0(A)$; in this case one denotes $E_{C,(A,V)}$ simply by $E_{C,A}$. If instead $(A,V)$ is non-complete, the vector bundle constructed as universal extension of $E_{C,(A,V)}$ is naturally isomorphic to $E_{C,A}$, as one can easily check by the very definition of Lazarsfeld-Mukai bundles; in other words, the universal extension looks as follows:
\begin{equation}\label{univ}
0\to H^1(E_{C,(A,V)})\otimes \oo_S\to E_{C,A}\to E_{C,(A,V)}\to 0
\end{equation}
with cocyle $\mathrm{id}\in \Hom(H^1(E_{C,(A,V)}),H^1(E_{C,(A,V)}))$.

The following remark will be useful in the sequel:
\begin{rem}\label{sug}
Let $E_{C,(A,V)}$ be a Lazarsfeld-Mukai bundle on a $K3$ surface $S$ and assume there is a surjective morphism $E_{C,(A,V)}\to N\otimes I_\xi$ for some line bundle $N\in\Pic(S)$ and some $0$-dimensional subscheme $\xi\subset S$.  Since $E_{C,(A,V)}$ is globally generated off a finite set and $h^2(E_{C,(A,V)})=0$, the line bundle $N$ shares the same properties and in particular $N$ is nontrivial. By \cite[Cor. 3.2]{SD}, line bundles on $K3$ surfaces have no base points outside their fixed components and thus $N$ is globally generated. 

This observation can be generalized to higher rank torsion free sheaves using \cite[Lem. 3.3]{LC2} in the following way. Let $Q$ be a torsion free sheaf on $S$ endowed with a surjection $E_{C,(A,V)}\to Q$. The sheaf $Q$ is then globally generated off a finite set and satisfies $h^2(Q)=0$. We can thus apply \cite[Lem. 3.3]{LC2} stating that such a $Q$ satisfies $h^0(\det Q)\geq 2$.   If the $K3$ surface $S$ contains no ($-2$)-curves then $\det Q$ has no fixed components and hence is globally generated again by  \cite[Cor. 3.2]{SD}. 
\end{rem}
For any $r,d$ we denote by $\G^r_d(|C|)$ the variety parametrizing pairs $(C',(A',V'))$ such that $C'\subset S$  is a smooth curve linearly equivalent to $C$, and $(A',V')\in G^r_d(C')$; there is a natural forgetful map $\pi:\G^r_d(|C|)\to |C|$.

Lazarsfeld-Mukai bundles are used in order to control the injectivity of the Petri map. 
\begin{prop}\label{pare}
If $C$ is general in its linear system and $(A,V)\in G^r_d(C)$ is base point free, then:
$$
\dim\ker\mu_{0,V}=h^0(E_{C,(A,V)}^\vee\otimes\omega_C\otimes A^\vee)-1.
$$
If moreover $(A,V)$ is complete, then 
$$
\dim\ker\mu_{0,A}=h^0(E_{C,A}^\vee\otimes E_{C,A})-1
$$
and  $\mu_{0,A}$ is injective if and only if $E_{C,A}$ is simple.
\end{prop}
\begin{proof}
The statement is proved for complete linear series in \cite{P}. We briefly sketch the proof in order to convince the reader that it works for non-complete linear series, as well. The kernel of $\mu_{0,V}$ is isomorphic to $H^0(C,M_{A,V}\otimes \omega_C\otimes A^\vee)$, where $M_{A,V}$ is the kernel of the evaluation map on the curve $V\otimes \oo_C\to A$. On the other hand, one has the following exact sequence:
\begin{equation}\label{ker}
0\to \oo_C\to E_{C,(A,V)}^\vee\otimes\omega_C\otimes A^\vee \to M_{A,V}\otimes \omega_C\otimes A^\vee\to 0.
\end{equation}
If $C$ is general in its linear system, the latter remains exact when we pass to global section. Indeed, the vanishing of the coboundary map $\delta: H^0(M_{A,V}\otimes \omega_C\otimes A^\vee)\to H^1(\oo_C)$ turns out to be equivalent to the surjectivity of the differential of the projection map $\pi:\mathcal G^r_d(|L|)\to|L|$ at the point $(C,(A,V))$; the result thus follows from Sard's Lemma. The last part of the statement follows tensoring \eqref{LM} with $E_{C,(A,V)}^\vee$ and only holds for complete linear series as it requires $h^1(E_{C,(A,V)})=0$.
\end{proof}
We now recall the structure of the Lazarsfeld-Mukai bundle associated with a linear series that is obtained restricting a line bundle $M\in\Pic(S)$ to a curve $C\subset S$. 
\begin{lem}[\cite{LC3} Lemma 4.1]\label{rest}
Let $N\in \Pic(S)$ satisfy $h^0(N)\geq 2$ and $h^1(N)=0$ ; also assume that $M:=\oo_S(C)\otimes N^\vee$ is globally generated and satisfies $h^1(M)=0$. Then the Lazarsfeld-Mukai bundle $E_{C,M\otimes \oo_C}$ sits in the following short exact sequence
\begin{equation}\label{res}
0\to N\to E_{C, M\otimes \oo_C}\to E_{D,\omega_D}\to 0,
\end{equation}
where $D$ is any smooth curve in the linear system $|M|$.
\end{lem}
We will also need the following:
\begin{rem}[\cite{LC3} Remark 6]\label{sei} Assume there exists a line bundle $N\in\Pic(S)$ such that $h^0(N)\geq 2$ together with an injective morphism $N\hookrightarrow E_{C,A}$ to some Lazarsfeld-Mukai bundle $E_{C,A}$. Then, the linear series $\vert A\vert$ is contained in $\vert (L\otimes N^\vee)\otimes\oo_C\vert$; this coincides with the restriction of $\vert L\otimes N^\vee\vert$ to $C$ if $h^1(N)=0$.
\end{rem}

Concerning the Lazarsfeld-Mukai bundle associated with the canonical line bundle, we state the following:
\begin{lem}\label{cano}
Let $E_{D,\omega_D}$ be the Lazarsfeld-Mukai bundle associated with the canonical line bundle on a smooth irreducible curve $D\subset S$. Then the following hold:
\begin{itemize}
\item[(i)] $E_{D,\omega_D}$ is simple;
\item[(ii)] $E_{D,\omega_D}$ does not depend on the choice of $D$ in its linear system.
\end{itemize}
\end{lem}
\begin{proof}
Sequence \eqref{ker} along with  the obvious vanishing $0=\ker\mu_{0,\omega_D}\simeq H^0(M_{\omega_D})$ implies that $\Hom(E_{D,\omega_D}, \oo_D)=0$; hence, (i) follows from \eqref{LM}. Having fixed $D$, we consider the Grassmannian $$G\left(g(D), H^0(E_{D,\omega_D})\right)\simeq \mathbb P(H^0(E_{D,\omega_D})^\vee)\simeq \mathbb P^{g(D)}.$$ For a general $\Lambda\in G\left(g(D), H^0(E_{D,\omega_D})\right)$ the cokernel of the evaluation $\Lambda\otimes\oo_S\to E_{D,\omega_D}$ is isomorphic to $\oo_{D_1}$ for some smooth curve $D_1\in |D|$; hence, $E_{D,\omega_D}\simeq E_{D_1,\omega_{D_1}}$. The rational map $h:G\left(g(D), H^0(E_{D,\omega_D})\right)\dashrightarrow |D|\simeq \mathbb P^{g(D)}$ constructed in this way is injective since $E_{D,\omega_D}$ is simple. Hence, it is birational and its image coincides with the open subset of  $|D|$ parametrizing smooth and irreducible curves; this proves (ii).
\end{proof}

\begin{rem}\label{canonical}
By \cite{Mu}, the moduli space $Sp(c(E_{D,\omega_D}))$ of sheaves on $S$ with the same Chern classes as $E_{D,\omega_D}$ is smooth of dimension $0$; our remark is equivalent to the statement that $Sp(c(E_{D,\omega_D}))$ only  contains one Lazarsfeld-Mukai bundle, namely, $E_{D,\omega_D}$ itself.
\end{rem}

\section{The Gieseker-Petri locus in genus $18$}\label{quattro}
In this section we prove the following theorem that clearly implies Theorem \ref{explicit}.

\begin{thm}\label{k3}
There exists a smooth $K3$ surface $S\subset \mathbb{P}^3$ such that $\Pic(S)=\Z H\oplus\Z C$, where $H$ is a hyperplane section of $S$ and $C$ is a smooth curve of genus $18$ and degree $16$. 

If $C$ is general in its linear system, then all the Brill-Noether varieties $G^3_{17}(C)$, $G^2_d(C)$ for $14\leq d\leq 17$ and $G^1_k(C)$ for $10\leq k\leq 17$ are smooth of the expected dimension.
\end{thm}

We recall Mori's Theorem (cf. \cite{mori}) stating that, if $e>0$ and $g\geq 0$ there is a smooth genus $g$ curve $C$ lying on a smooth quartic surface $S\subset \mathbb P^3$ such that $\deg (C)=e$ if and only if  either $g=e^2/8+1$, or $g<e^2/8$ and $(e,g)\neq (5,3)$. This ensures the existence of a $K3$ surface $S$ as in the statement since $g=18<(\deg (C))^2/8=(H\cdot C)^2/8=32$. Using the intersection numbers $H^2=4$, $H\cdot C=16$ and $C^2=2g-2=34$, one can easily verify that $S$ contains neither curves of genus $1$ nor ($-2$)-curves, or equivalently (cf. \cite{gabi1}), that $0$ and $-1$ are not represented by the quadratic form
\begin{equation}\label{diofantea}
Q(a,b):=2a^2+16ab+17b^2.
\end{equation}
\begin{rem}\label{ample}In particular, any effective line bundle $L$ on our $K3$ surface $S$ satisfies $c_1(L)^2>0$ and is globally gene\-rated by \cite{SD}. Something even stronger holds, namely, $L$ is automatically ample (cf., e.g., \cite[Corollary 8.1.6]{Hu}).\end{rem}

From now on, we assume $C$ to be general in its linear system , so that we can apply Proposition \ref{pare} in order to translate the injectivity  of Petri maps on $C$ in terms of Lazarsfeld-Mukai bundles on $S$. We will study the simplicity of such bundles by analyzing their slope-stability with respect to the line bundle $\oo_S(C-H)$.  

\begin{lem}\label{ch}
Let $S$ be a $K3$ surface as in Theorem \ref{k3}.
Then, the line bundle $\oo_S(C-H)$ is ample and the following hold:
\begin{itemize}
\item[(i)] the slope of any line bundle on $S$ with respect to $C-H$ is divisible by $6$;
\item[(ii)] if an effective line bundle $L\in\Pic(S)$ satisfies $\mu_{C-H}(L)=6$, then $c_1(L)=C-H$;
\item[(iii)] if an effective line bundle$L\in\Pic(S)$ satisfies $\mu_{C-H}(L)=12$, then \\$c_1(L)\in \left\{H,2(C-H), 4C-5H\right\}$.
\end{itemize}
\end{lem}
\begin{proof}
Since $(C-H)^2>0$ and $H\cdot(C-H)>0$, then $\oo_S(C-H)$ is effective and thus automatically ample by Remark. Item (i) follows trivially from the intersection numbers $(C-H)\cdot C=18$ and $(C-H)\cdot H=12$. Now let $L$ be an effective line bundle (hence, $c_1(L)^2>0$) and write  $c_1(L)=aH+bC$ for some integers $a$ and $b$. First assume that $c_1(L)\cdot (C-H)=6$ and $c_1(L)\neq C-H$. Since $(C-H)^2=6$, the Hodge Index Theorem yields either $c_1(L)^2=2$ or $c_1(L)^2=4$. The former case does no occur since $1$ is not represented by the quadratic form \eqref{diofantea}; the latter case can also be excluded since the system of diophantine equations $2a^2+17b^2+16ab-2=12a+18b-6=0$ has no integral solutions. This proves (ii). 

We now assume $c_1(L)\cdot (C-H)=12$ as in (iii), or equivalently, $a=1+3k$, $b=-2k$ with $k\in\mathbb Z$. This contradicts the inequality $c_1(L)^2>0$ unless $k\in \{-2,-1,0\}$, thus proving (iii).
\end{proof}

We recall that any slope-stable (with respect to any polarization) coherent sheaf $E$ on $S$ moves in a smooth moduli space of dimension
\begin{equation}\label{dim}
(1-\rk E)c_1(E)^2+2\rk E c_2(E)-2(\rk E)^2+2,
\end{equation}
cf. \cite{Mu}; the Chern classes of $E$ thus satisfy the inequality 
\begin{equation}\label{Bog+}
c_2(E)\geq -\frac{1}{\rk E}+\rk E+\frac{\rk E-1}{2\rk E}c_1(E)^2,
\end{equation}
that is slightly stronger than Bogomolov's inequality.

First of all, we study complete pencils on a curve $C\subset S$ as in Theorem \ref{k3}.
\begin{prop}\label{pencils}
Let $S\subset \mathbb{P}^3$ be a $K3$ surface as in Theorem \ref{k3}. If $C$ is general in its linear system, then $C$ has maximal gonality $10$ and, for $10\leq k\leq 17$, the Brill-Noether variety $G^1_k(C)$ is smooth at all points corresponding to complete pencils.
\end{prop}

\begin{proof}
By Theorem 3 in \cite{gabi1}, $C$ has maximal gonality $10$. Let $A$ be a complete $g^1_k$ on $C$ with $10\leq k\leq 17$. By induction on $k$, we may assume $A$ is base point free. By contradiction, we suppose that the rank $2$ Lazarsfeld-Mukai bundle $E=E_{C,A}$ is non-simple. Hence, it cannot be $\mu_{C-H}$-stable and there is a destabilizing short exact sequence:
\begin{equation}\label{seminario}
0\to M\to E\to N\otimes I_{\xi}\to 0,
\end{equation}
where $N, M\in\Pic(S)$ satisfy 
\begin{equation}\label{pescara}
\mu_{C-H}(M)\geq\mu_{C-H}(E)=9\geq\mu_{C-H}(N)>0,
\end{equation} 
with the last inequality following from the fact that $N$ is globally generated and non-trivial by Remark \ref{sug}. By Lemma \ref{ch} (i)-(ii), the only possibility is $c_1(N)=C-H$ and $c_1(M)=H$. Since $(C-2H)^2<0$, then both $\mathrm{Hom}(M,N\otimes I_\xi)=0$ and $\mathrm{Hom}(N\otimes I_\xi,M)=0$. The non-simplicity of $E$ thus yields $E\simeq \oo_S(H)\oplus\oo_S(C-H)$. 

We consider the rational map $h_E:G(2,H^0(E))\dashrightarrow \G^1_k(\vert C\vert)$ mapping a general $2$ dimensional subspace $\Lambda\subset H^0(E)$ to the pair $(C_\Lambda,A_\Lambda)$, where $C_\Lambda$ is the degeneracy locus of the (injective) evaluation map $ev_\Lambda:\Lambda\otimes\oo_S\to E$ and $\omega_{C_\Lambda}\otimes A_\Lambda^\vee$ is the cokernel of $ev_\Lambda$. The fiber of $h_E$ over $(C,A)\in \im\,h_E\subset\G^1_k(\vert C\vert)$ is isomorphic to $$\mathbb{P}\Hom(E,\omega_C\otimes A^\vee)\simeq \mathbb{P}H^0(S,E\otimes E^\vee),$$ which is $1$-dimensional. It follows that 
$$
\dim\im\,h_{E}= 2(h^0(E)-2)-1=2(g-k+1)-1<g,$$ as $k\geq 10>(g+1)/2$; in particular, the image of $h_E$ does not dominate the linear system $|C|$. This implies that, if $C$ is general in its linear system and $10\leq k\leq 17$, the Lazarsfeld-Mukai bundle associated with {\em any} complete, base point free $g^1_k$ on $C$ is simple, and thus the statement follows from Proposition \ref{pare}.
\end{proof}
  
 We now treat complete linear series of type $g^2_d$.

\begin{prop}\label{nets}
Let $S\subset \mathbb{P}^3$ be a $K3$ surface as in Theorem \ref{k3}. If $C$ is general in its linear system, then $C$ has no linear series of type $g^2_{13}$. Furthermore, for $14\leq d\leq 17$ the Brill-Noether variety $G^2_d(C)$ is smooth at all points parametrizing complete nets .
\end{prop}

\begin{proof}
Let $C$ be general in its linear system and $A\in \Pic^d(C)$ be a complete $g^2_d$ on $C$ with $d\leq 17$; by induction on $d$, we may assume it to be base point free. By contradiction, suppose that the rank $3$ Lazarsfeld-Mukai bundle $E=E_{C,A}$ is non-simple, and hence not $\mu_{C-H}$-stable. We separately analyze two cases.\\\vspace{0.1cm}

\noindent{\em CASE A: The maximal destabilizing subsheaf of $E$ is a $\mu_{C-H}$-stable rank $2$ vector bundle $E_1$.}

We consider the short exact sequence
\begin{equation}\label{2A}
0\to E_1\to E\to N\otimes I_\xi\to 0,
\end{equation} 
where $\xi\subset S$ is a $0$-dimensional subscheme, $N\in \Pic(S)$ is globally generated and non-trivial by Remark \ref{sug}, and the following inequalities are satisfied: $$\mu_{C-H}(E_1)\geq\mu_{C-H}(E)=6\geq \mu_{C-H}(N)>0.$$ Lemma \ref{ch} (i)-(ii)  yields $c_1(N)=C-H$ and $c_1(E_1)=H$. Since $\mu_{C-H}(E_1)=\mu_{C-H}(N)$ and $E_1$ is stable, then $\mathrm{Hom}(E_1,N\otimes I_\xi)=\mathrm{Hom}(N\otimes I_\xi,E_1)=0$ (cf. \cite{Fr}). As $E$ is non-simple, then $\xi=\emptyset$ and \eqref{2A} splits, that is, $E\simeq E_1\oplus \oo_S(C-H)$. By Remark \ref{sei}, the linear series $|A|$ is then contained in the restriction of $|H|$ to $C$ and thus $d\leq H\cdot C=16$.

We perform a parameter count like in \cite{LC2} contradicting the generality of $C$. The stable sheaf $E_1$ moves in a moduli space $\mathcal M_1$ of dimension $4d-58$, cf. \eqref{dim}. Let $\mathcal M_1^\circ$ denote the oben subset of $\mathcal M_1$ parametrizing generically generated vector bundles with vanishing $H^1$ and $H^2$, and let $p:\mathcal G_1\to \mathcal M_1^\circ$ be the Grassmann bundle  whose fiber over a point $[E_1]\in  \mathcal M_1^\circ$ is $G(3, H^0(E_1\oplus \oo_S(C-H)))$. We define a rational map $h_1:\mathcal G_1\dashrightarrow \mathcal G^2_d(|C|)$ mapping a general point $(E_1\oplus \oo_S(C-H),\Lambda)\in\G_1$ to the pair $(C_\Lambda,A_\Lambda)$, where $C_\Lambda$ is the degeneracy locus of the evaluation map $$ev_\Lambda:\Lambda\otimes\oo_S\to E_1\oplus \oo_S(C-H),$$ which is injective for a general $\Lambda\in G(3,H^0(S,E_1\oplus \oo_S(C-H)))$, and $\omega_{C_\Lambda}\otimes A_\Lambda^\vee$ is the cokernel of $ev_\Lambda$. A general fiber of $p$ has dimension $60-3d$ and the fiber of $h_1$ over a general point $(C_\Lambda,A_\Lambda)\in \mathrm{Im} h_1$ is isomorphic to the projective line 
$$
\mathbb{P}\Hom(E_1\oplus \oo_S(C-H),\omega_C\otimes A^\vee)\simeq \mathbb{P}\Hom(E_1\oplus \oo_S(C-H),E_1\oplus \oo_S(C-H));
$$
thus, the image of $h_1$ has dimension $d+1\leq 17<g$ and does not dominate the linear system $|C|$.\\\vspace{0.1cm}

\noindent{\em CASE B: There is a line bundle $M\in \Pic(S)$ destabilizing $E$ and having maximal slope.}

Having maximal slope, the line subbundle $M\subset E$ is saturated and we have a short exact sequence
\begin{equation}\label{2B}
0\to M\to E\to E/M\to 0,
\end{equation} 
where $E/M$ is a rank $2$ torsion free sheaf such that $$\mu_{C-H}(M)\geq \mu_{C-H}(E)=6\geq \mu_{C-H}(E/M).$$
The line bundle $\det E/M$ is globally generated and non-trivial by Remark \ref{sug} and thus satisfies $0<\mu_{C-H}(\det E/M)=2\mu_{C-H}(E/M)\leq 12$. In particular, by Lemma \ref{ch}(i) either $\mu_{C-H}(\det E/M)=6$, or $\mu_{C-H}(\det E/M)=12$.\\\vspace{0.05cm}

\noindent{\em SUBCASE B1: The bundle $E/M$ in \eqref{2B} satisfies $\mu_{C-H}(\det E/M)=6$.}

Lemma \ref{ch}(ii) yields $c_1(E/M)=C-H$ and $c_1(M)=H$. Since $E/M$ is generically generated, $H^2(E/M)=0$ and $\mu_{C-H}(E/M)=3$, then $E/M$ is $\mu_{C-H}$-stable by Lemma \ref{ch}(i). The inequality $\mu_{C-H}(E/M)<\mu_{C-H}(M)$ implies $\mathrm{Hom}(M,E/M)=0$. We now show that $\mathrm{Hom}(E/M,M)=0$, too. By contradiction, assume the existence of a non-zero morphism $\alpha:E/M\to M$. The image of $\alpha$ equals $\mathcal{O}_S(H-D)\otimes I_\xi$ for some effective divisor $D$ such that $\mathcal{O}_S(H-D)$ is globally generated and some $0$-dimensional subscheme $\xi\subset S$. The stability of $E/M$ yields
$$
3=\mu_{C-H}(E/M)< \mu_{C-H}(H-D))\leq \mu_{C-H}(H)=12.
$$
Since $\oo_S(2H-C)$ is non-effective, Lemma \ref{ch}(ii) implies that $D=0$. Equivalently, $\mathrm{Im}\alpha\simeq \mathcal{O}_S(H)\otimes I_\xi$ and $\mathrm{ker}\alpha \simeq \mathcal{O}_S(C-2H)\otimes I_\eta$ for some $0$-dimensional subscheme $\eta\subset S$. One gets the contradiction
$$
d=c_2(E)=H\cdot (C-H)+c_2(E/M)\geq H\cdot (C-H)+H\cdot(C-2H)=20.
$$
Therefore, $\mathrm{Hom}(E/M,M)=0$ and the fact that $E$ is non-simple forces \eqref{2B} to split, that is, $E\simeq M\oplus E/M$. However, a parameter count as the one performed in Case A (using the fact that $E/M$ moves in a moduli space of dimension $4d-60$) shows that such a splitting Lazarsfeld-Mukai bundle cannot be associated with a general curve in the linear system $|C|$ as soon as $d\leq 17$.\\\vspace{0.05cm}

\noindent{\em SUBCASE B2: The bundle $E/M$ in \eqref{2B} satisfies $\mu_{C-H}(\det E/M)=12$.}

Equivalently, we have $\mu_{C-H}(M)=\mu_{C-H}(E)=\mu_{C-H}(E/M)=6$. Since $E/M$ is generically generated and $H^2(E/M)=0$ by Remark \ref{sug}, it is $\mu_{C-H}$-semistable by Lemma \ref{ch}(i). More strongly, Lemma \ref{ch}(ii) ensures that $E/M$ is  $\mu_{C-H}$-stable as soon as the vani\-shing \mbox{$\Hom(E/M,\oo_S(C-H))=0$} holds.

By contradiction, 
let $\alpha\in\mathrm{Hom}(E/M,\oo_S(C-H))$ be non-zero. The semistabily of $E/M$ yields $c_1(\mathrm{Im}\alpha)=C-H$. One gets a short exact sequence
$$
0\to \det E/M \otimes (H-C)\otimes I_\eta\to E/M\stackrel{\alpha}{\longrightarrow} (C-H)\otimes I_\xi\to 0$$
for some $0$-dimensional subschemes $\xi$ and $\eta$, and thus
$$
c_2(E/M)\geq c_1(E/M)\cdot (C-H)-6=6;
$$
hence, by \eqref{2B}, we obtain 
\begin{equation}\label{grado}
d=c_1(M)\cdot c_1(E/M)+c_2(E/M)\geq c_1(E/M)\cdot (C-c_1(E/M))+6.
\end{equation}
On the other hand, Lemma \ref{ch}(iii) yields $c_1(E/M)\in \left\{H,2(C-H), 4C-5H\right\}$. In all the three cases inequalities $d\leq 17$ and \eqref{grado} are in contradiction. This proves that $E/M$ is $\mu_{C-H}$-stable and hence $\Hom(M,E/M)=\Hom (E/M,M)=0$. Since $E$ is non-simple, then 
$E\simeq M\oplus E/M$ and one falls under Case A. \end{proof}

The next step consists in studying linear series of type $g^3_d$.
\begin{prop}\label{prop:iuppi}
Let $S\subset \mathbb{P}^3$ be a $K3$ surface as in Theorem \ref{k3}. If $C$ is general in its linear system, then the following hold:
\begin{itemize}
\item[(i)] the Brill-Noether variety $G^3_{17}(C)$ is smooth of the expected dimension; 
\item[(ii)] the Brill-Noether variety $G^3_{16}(C)$ consists of a unique isolated point corresponding to the line bundle $\oo_C(H)$.
\end{itemize}
\end{prop}
We will first prove the following weaker result:
\begin{lem}\label{intermedio}
Let $S\subset \mathbb{P}^3$ be a $K3$ surface as in Theorem \ref{k3}. If $C$ is general in its linear system, then the Petri map associated with any base point free $g^3_{17}$ on $C$ is injective and the only $g^3_{16}$ is $\oo_C(H)$.
\end{lem}
\begin{proof}
It is enough to consider complete linear series of type $g^3_d$ for $d=16,17$. Indeed, if $C$ admits a $g^r_d$ with $r\geq 4$ and $d=16,17$, then it is easy to show that it admits a positive dimensional family of complete (but not necessarily base point free) $g^3_{16}$. Furthermore, instead of considering complete $g^3_{16}$ and $g^3_{17}$ with base points, we will study complete, base point free linear series of type $g^3_d$ for all values of $d\leq 17$. 

Let $E:=E_{C,A}$ be a non-simple rank $4$ Lazarsfeld-Mukai bundle associated with a complete base point free $A$ on $C$ of type $g^3_d$ for $d\leq17$ such that the Petri map $\mu_{0,A}$ is non-injective (the last request is automatically satisfied if $d\leq16$). Since $\mu_{C-H}(E)=9/2$, Lemma \ref{ch}(i) excludes that $E$ is destabilized by a vector subbundle of rank $3$. Hence, only two cases need to be taken in consideration. \\\vspace{0.1cm}

\noindent{\em CASE A: The maximal destabilizing subsheaf of $E$ is a $\mu_{C-H}$-stable rank $2$ vector bundle $E_1$.} 

We have a short exact sequence:
\begin{equation}\label{3A}
0\to E_1\to E\to E_2\to 0,
\end{equation}
where $E_2$ is a torsion free sheaf of rank $2$ satisfying $\mu_{C-H}(E_2)\leq \mu_{C-H}(E)=9/2$. Again Lemma \ref{ch}(i) forces $E_2$ to be stable. Therefore, $c_2(E_i)\geq \frac{3}{2}+\frac{1}{4}c_1(E_i)^2$ for $i=1,2$ by \eqref{Bog+}. Furthermore, $\det E_2$ is globally generated and non-trivial by Remark \ref{sug} and its slope is bounded above by $2\mu_{C-H}(E)=9$. Lemma \ref{ch}(i)-(ii) thus implies $c_1(E_2)=C-H$ and $c_1( E_1)=H$.  One gets the contradiction
$$
17\geq d=c_2(E)=H\cdot(C-H)+c_2(E_1)+c_2(E_2)\geq 15+\frac{1}{4}(H^2+(C-H)^2)=\frac{35}{2}.
$$\vspace{0.1cm}

\noindent{\em CASE B: There is a line bundle $N\in\Pic(S)$ destabilizing $E$ and having maximal slope.}

The line bundle $N$ is a saturated subsheaf of $E$ and thus sits in a short exact sequence
\begin{equation}\label{3B}
0\to N\to E\to E/N\to 0,
\end{equation}
where $E/N$ is a torsion free sheaf of rank $3$ such that $\mu_{C-H}(E/N)\leq \mu_{C-H}(E)=9/2$. By Remark \ref{sug}, $\det E/N$ is a non-trivial globally generated line bundle whose slope is bounded above by $3\mu_{C-H}(E)=27/2$. Lemma \ref{ch} yields that either $c_1(E/N)=C-H$ or $\mu_{C-H}(\det E/N)=12$. \\\vspace{0.05cm}

\noindent{\em SUBCASE B1 :  The sheaves in \eqref{3B} satisfy $c_1(N)=H$ and $c_1(E/N)=C-H$.}

 Since $\mu_{C-H}(E/N)=2$, then $E/N$ is stable by Lemma \ref{ch} (i) (as it cannot admit a destabilizing quotient sheaf of smaller rank for slope reasons). In particular, from \eqref{Bog+} we get $c_2(E/N)\geq \frac{8}{3}+\frac{1}{3}c_1(E/N)^2=\frac{14}{3}$, and hence:
$$
17\geq d=c_2(E/N)+c_1(N)\cdot c_1(E/N)=c_2(E/N)+H\cdot (C-H)\geq 12+\frac{14}{3}.$$ The only possibility is thus $c_2(E/N)=5$ and $d=17$.  We first show that, if $C$ is general in its linear system, then \eqref{3B} cannot split. The sheaf $E/N$ moves in a moduli space $\mathcal M$ of dimension $2$, cf. \eqref{dim}; let $\mathcal M^\circ$ be its open subset parametrizing generically generated torsion free sheaves with vanishing $H^1$ and $H^2$. Over $\mathcal M^\circ$ we consider the Grassmann bundle $\G$ whose fiber over a general $[F]\in \mathcal M^\circ$ is the $16$-dimensional Grassmannian $G(4, H^0(\oo_S(H)\oplus F))$.  It is enough to remark that the image of the rational map \mbox{$h:\mathcal G \dashrightarrow\G^4_{17}(|C|)$} defined as in the proof of Proposition \ref{nets} does not dominate $|C|$; this follows because $\dim\,\G=18$ and the fibers of $h$ have positive dimension. 

Hence, \eqref{3B} does not split. Note that $\Hom(N,E/N)=0$ as $E/N$ is $\mu_{C-H}$-stable. The non-simplicity of $E$ implies the existence of a morphism $0\neq \alpha: E/N\to N\simeq \mathcal{O}_S(H)$. Write $\mathrm{Im}\,\alpha=\mathcal{O}_S(H-D)\otimes I_\xi$ for some effective divisor $D$ and $0$-dimensional subscheme $\xi$. As $E/N$ is $\mu_{C-H}$-stable and globally generated, then $\mathcal{O}_S(H-D)$ is a globally generated line bundle satisfying 
$$
2=\mu_{C-H}(E/N)< \mu_{C-H}(H-D))\leq \mu_{C-H}(H)=12.
$$
By Lemma \ref{ch}, either $D=0$ or $H-D\equiv C-H$; the latter case can be excluded since it implies $D\sim 2H-C$, which is not effective. We conclude that $D=0$ and get the following short exact sequence
$$
0\to \mathcal{O}_S(C-2H)\otimes I_\eta\to E/N \stackrel{\alpha}{\to} \oo_S(H)\otimes I_\xi\to 0$$ for some $0$-dimensional subschemes $\xi,\eta\subset S$. This leads to the contradiction 
$5=c_2(E/N)\geq H\cdot (C-2H)=8$.\\\vspace{0.05cm}

\noindent{\em SUBCASE B2 :  The sheaf $E/N$ in \eqref{3B} satisfies $\mu_{C-H}(\det E/N)=12$.}

We apply Lemma \ref{ch}(iii). The case $c_1(E/N)=4C-5H$ does not occur because it would imply $c_1(N)=5H-3C$ and thus the contradiction $$d=c_2(E)\geq (4C-5H)\cdot (5H-3C)=52.$$ Now, assume $c_1(E/N)=2(C-H)$ and $c_1(N)=2H-C$. It follows that 
\begin{equation}\label{stanca}
c_2(E/N)=d-2(C-H)(2H-C)=d-12\leq 5.
\end{equation} 
In particular, $E/N$ cannot be $\mu_{C-H}$-stable because otherwise \eqref{Bog+} would imply $c_2(E/N)\geq \frac{8}{3}+\frac{1}{3}c_1(E/N)^2=\frac{32}{3}$.  However, since $\mu_{C-H}(E/N)=4$ and $h^2(E/N)=0$, Lemma \ref{ch}(i) excludes that $E/N$ is destabilized by any subsheaf of rank $2$. The sheaf $E/N$ is thus destabilized by a subsheaf $M$ of maximal slope and rank $1$ and the quotient $Q:=(E/N)/M$ is a generically generated torsion free sheaf of rank $2$ sa\-ti\-sfying $H^2(Q)=0$. By Remark \ref{sug}, $\det Q$ is a non-trivial globally generated line bundle such that $\mu_{C-H}(\det Q)=2\mu_{C-H}(Q)\leq 2\mu_{C-H}(E/N)=8$; again Lemma \ref{ch}(i)-(ii) yields $c_1(Q)=C-H=c_1(M)$ and $c_2(E/N)\geq (C-H)^2=6$, contradicting \eqref{stanca}. This excludes the case  $c_1(E/N)=2(C-H)$.

It remains to consider the case $c_1(E/N)=H$ and $c_1(N)=C-H$. Remark \ref{sei} then implies that the linear series $|A|$ is contained in $|\mathcal{O}_C(H)|$, which is a complete base point free $g^3_{16}$. The only possibility is thus $A\simeq \mathcal{O}_C(H)$.
\end{proof}

\begin{proof}[Proof of Proposition \ref{prop:iuppi}]
Lemma \ref{intermedio} implies that any linear series of type $g^3_{16}$ or $g^3_{17}$ is complete. Furthermore, $G^3_{17}(C)$ is smooth at all points corresponding to base point free linear series. In order to conclude, it remains to show that $\dim\ker\mu_{0,  \mathcal{O}_C(H)}=2$, or, equivalently by Proposition \ref{pare}, that the Lazarsfeld-Mukai bundle $E=E_{C, \mathcal{O}_C(H)}$ satisfies $h^0(E\otimes E^\vee)=3$. By Lemma \ref{rest}, $E$ sits in the short exact sequence
\begin{equation}\label{london}
0\to \mathcal{O}_S(C-H)\to E\to E_{H,\omega_H}\to 0,
\end{equation}
where $E_{H,\omega_H}$ is the rank $3$ Lazarsfeld-Mukai bundle associated with the canonical sheaf $\omega_H$ on any smooth hyperplane section $H$ of $S$ (cf. Lemma \ref{cano}). 

We first claim that $E_{H,\omega_H}$ is $\mu_{C-H}$-stable. As $E_{H,\omega_H}$ is globally generated and satisfies $H^2(E_{H,\omega_H})=0$ and $\mu_{C-H}(E_{H,\omega_H})=4$, then Lemma \ref{ch}(i) implies that it cannot be destabilized by any vector bundle of rank $2$. If it is not stable, there exists a destabilizing sequence
$$0\to N\to E_{H,\omega_H}\to E_{H,\omega_H}/N\to 0,$$
with $N\in \Pic(S)$ and $Q:=E_{H,\omega_H}/N$ a globally generated torsion free sheaf of rank $2$ satisfying $h^2(Q)=0$. In particular, the line bundle $\det Q$ is globally generated and non-trivial by Remark \ref{sug} and satisfies
$$
0<\mu_{C-H}(\det Q)=2\mu_{C-H}(Q)\leq 2\mu_{C-H}(E_{H,\omega_H})=8;
$$
hence, $c_1(Q)=C-H$ by Lemma \ref{ch}(i)-(ii) and $c_1(N)=2H-C$. One gets
$$
4=c_2(E_{H,\omega_H})=(C-H)\cdot (2H-C)+c_2(Q)=6+c_2(Q),
$$ 
and this is a contradiction since the second Chern class of a globally generated rank $2$ torsion free sheaf on $S$ is always  positive. Therefore, $E_{H,\omega_H}$ is $\mu_{C-H}$-stable as claimed.

By applying first $\Hom(E,-)$ and then $\Hom(-,\mathcal{O}_S(C-H))$ and $\Hom(-,E_{H,\omega_H})$ to \eqref{london} and by remarking that $\Hom(\mathcal{O}_S(C-H),E_{H,\omega_H})=0$ for slope reasons, one shows that
$$
2\leq \dim\ker\mu_{0,  \mathcal{O}_C(H)}= h^0(S,E\otimes E^\vee)-1\leq 1+\dim\Hom(E_{H,\omega_H},\mathcal{O}_S(C-H)),
$$
and the inequality is strict unless the sequence (\ref{london}) splits. Therefore, if we prove that $\dim\Hom(E_{H,\omega_H},\mathcal{O}_S(C-H))\leq1$, then $E\simeq\mathcal{O}_S(C-H)\oplus E_{H,\omega_H}$ and $\dim\ker\mu_{0,  \mathcal{O}_C(H)}=2$, as desired. Given $0\neq \alpha:E_{H,\omega_H}\to \mathcal{O}_S(C-H)$, there exist an effective divisor $D$ and a $0$-dimensional subscheme $\xi\subset S$ such that $\mathrm{Im}\,\alpha=\oo_S(C-H-D)\otimes I_{\xi}$. The line bundle $\oo_S(C-H-D)$ is globally generated and its slope is bounded below by $\mu_{C-H}(E_{H,\omega_H})=4$ and above by $\mu_{C-H}(C-H)=6$. Lemma \ref{ch} thus yields $D=0$ and $E_{H,\omega_H}$ sits in the following short exact sequence:
\begin{equation}\label{sonno}
0\to K\to E_{H,\omega_H}\to \oo_S(C-H)\otimes I_{\xi}\to 0,
\end{equation}
where $K$ is a vector bundle of rank $2$ such that $c_1(K)=2H-C$, $\chi(K)=l(\xi)-1$ and $c_2(K)=-2-l(\xi)$. Moreover, $K$ is $\mu_{C-H}$-stable because otherwise it would be destabilized by a line bundle $N$  such that
$$
3=\mu_{C-H}(K)\leq\mu_{C-H}(N)<\mu_{C-H}(E_{H,\omega_H})=4,
$$
thus contradicting Lemma \ref{ch}(i). The stability of $K$ implies $c_2(K)\geq\frac{3}{2}+\frac{1}{4}c_1(K)^2=-2$ by \eqref{Bog+},
hence $l(\xi)=0$. By applying $\Hom(-,\oo_S(C-H))$ to the sequence \eqref{sonno}, one finds that $$\dim\Hom(E_{H,\omega_H},\oo_S(C-H))\leq 1+\dim\Hom(K,\oo_S(C-H)).$$ We will now show that $\Hom(K,\oo_S(C-H))=0$, which concludes the proof. If there exists $0\neq \beta:K\to\oo_S(C-H)$, then $\im\,\beta=\oo_S(C-H-D_1)\otimes I_{\xi_1}$ for some divisor $D_1\geq 0$ and $0$-dimensional subscheme $\xi_1\subset S$. Since $K$ is stable, then 
$$3=\mu_{C-H}(K)\leq \mu_{C-H}(C-H-D_1)\leq\mu_{C-H}(C-H)=6,$$  thus $D_1=0$ by Lemma \ref{ch}(i) and one gets the following short exact sequence:
$$
0\to \oo_S(3H-2C)\to K\stackrel{\beta}{\to} \oo_S(C-H)\otimes I_{\xi_1}\to 0.
$$
One gets a contradiction since $-2=c_2(K)=(3H-2C)\cdot(C-H)+l(\xi_1)\geq 0$. This concludes the proof.
\end{proof}
In order to conclude the proof of Theorem \ref{k3}, it only remains to show that the Brill-Noether varieties of $C$ are smooth of the expected dimension at the points parametrizing non-complete linear series.

\begin{prop}\label{noncomplete}
Let $S\subset \mathbb{P}^3$ be a $K3$ surface as in Theorem \ref{k3}. If $C$ is general in its linear system, then any non-complete linear series on $C$ of degree $\leq g-1=17$ has injective Petri map.
\end{prop}

\begin{proof}
The only complete linear series on $C$ with non-injective Petri-map is $\oo_C(H)$. Therefore, it is enough to prove the statement  for non-complete linear series of the form $(\oo_C(H),V)$ with $\dim V=3$. Let $E_V:=E_{C,(\oo_C(H),V)}$ be the associated Lazarsfeld-Mukai bundle. By \S \ref{tre},  $E_V$ satisfies $h^1(E_V)=1$ and sits in the following universal extension, cf. \eqref{univ}:
\begin{equation}\label{uni}
0\to \oo_S\to E_{C, \mathcal{O}_C(H)}\to E_V\to 0.
\end{equation}
By Proposition \ref{pare}, we need to show that $\Hom(E_V, \oo_C(C-H))=0$. As an intermediate step, we will first prove that  $E_V$ is simple.                                                  
Short exact sequences \eqref{london} and \eqref{uni} fit in the following commutative diagram:
$$
 \xymatrix{
&&0&0\\
0\ar[r]&\mathcal{O}_S(C-H)\ar[r]&E_V\ar[u]\ar[r]&Q\ar[u]\ar[r]&0\\
0\ar[r]&\mathcal{O}_S(C-H)\ar@{=}[u]\ar[r]&E_{C,\mathcal{O}_C(H)}\ar[u]\ar[r]&E_{H,\omega_H}\ar[u]\ar[r]&0.\\
&&\oo_S\ar[u]\ar@{=}[r]&\oo_S\ar[u]\\
&&0\ar[u]&0\ar[u]\\
}
$$
From the right hand side of the diagram, one deduces that $Q$ coincides with the Lazarsfeld-Mukai bundle $E_W:=E_{H,(\omega_H,W)}$ associated with some non complete linear series $(\omega_H,W)$ on some hyperplane section $H$ of $S$. In particular, the bundle $E_W$ is globally generated and $\mu_{C-H}(E_W)=6$. If $E_W$ were not $\mu_{C-H}$-stable, by Lemma \ref{ch} it would lie in a short exact sequence
$$
0\to\mathcal{O}_S(2H-C)\to E_W\to \oo_S(C-H)\otimes I_\xi\to0,
$$
for some $0$-dimensional subscheme $\xi\subset S$, and thus the contradiction
$$
4=c_2(E_W)=(2H-C)\cdot (C-H)+l(\xi)\geq 6.
$$
Hence, $E_W$ is $\mu_{C-H}$-stable and $\Hom(\oo_S(C-H),E_W)=\Hom(E_W,\oo_S(C-H))=0$. It follows that $E_V$ is simple unless \mbox{$E_V\simeq E_W\oplus \oo_S(C-H)$}. We will now show that, if $C$ is general in its linear system, the Lazarsfeld-Mukai bundle $E_V$ associated with any non-complete linear series $(\oo_C(H),V)$ does not split in this way. 

Remember that $E_{H,\omega_H}$ is rigid by Remark \ref{canonical}. We consider the Quot-scheme $\mathcal Q:=\mathrm{Quot}_S(E_{H,\omega_H},P)$, where $\mathrm{P}$ is the Hilbert polynomial of $E_W$. It is well known (cf. \cite{lehn} Proposition 2.2.8) that, for any $[E_W]\in \mathcal Q$, the following holds:
\begin{equation}
\dim_{[E_W]}\mathcal Q\leq\dim\Hom(\oo_S,E_W)=h^0(E_W)=3;
\end{equation}
 and hence the dimension of any component of the Quot-scheme is $\leq 3$.
 
 Let $\G_\mathcal Q\to\mathcal Q$ be the Grassmann bundle whose fiber over a general $[E_W]\in \mathcal Q$ is the $15$-dimensional Grassmannian $G(3, H^0(E_W\oplus \mathcal{O}_S(C-H)))$. We define $h_{\mathcal Q}:\G_\mathcal Q\dashrightarrow |C|$ mapping a general point $(E_W,\Lambda)\in\G_\mathcal Q$ to the degeneracy locus of the evaluation map \mbox{$ev_\Lambda:\Lambda\otimes\oo_S\to E_W\oplus \mathcal{O}_S(C-H)$}, which is a smooth curve $C_\Lambda \in |C|$. The fibers of $h_{\mathcal Q}$ are at least $1$-dimensional because the composition of $ev_\Lambda$ with any autormophism of $ E_W\oplus \mathcal{O}_S(C-H)$ has the same degeneracy locus $C_\Lambda$. Therefore, the image of $h_\mathcal Q$ has dimension $\leq \dim \G_\mathcal Q-1=\dim \mathcal Q+15-1 \leq 17$ and $h_\mathcal Q$ is not dominant. This shows that, if $C$ is general, the Lazarsfeld-Mukai bundle $E_V$ associated with any non-complete linear series $(\oo_C(H),V)$ of type $g^2_{16}$ on $C$ is simple. 
 
 In order to conclude the proof, we now show that the simplicity of $E_V$ implies that $\dim\Hom(E_V, \oo_C(C-H))=1$ and thus the injectivity of the Petri map $\mu_{0,(\oo_C(H),V)}$. Consider the short exact sequence defining $E_V$:
 \begin{equation}\label{f}
 0\to V^\vee\otimes \oo_S\to E_V\stackrel{f}{\to} \oo_C(C-H)\to 0.
 \end{equation}
Let $f':E_V\to \oo_C(C-H)$ be a morphism different from $f$; we may assume that $f'$ is surjective since this is true for $f$ and thus for a general morphism from $E_V$ to $\oo_C(C-H)$.  We want to show that $f'$ is obtained by composing $f$ with an automorphism of $E_V$, and is thus a scalar multiple of $f$ since $E_V$ is simple. Equivalently, if we consider the long exact sequence 
$$
0\to H^0(E_V\otimes E_V^\vee)\to \Hom(E_V, \oo_C(C-H))\stackrel{\delta}{\to} \Ext^1(E_V,V^\vee\otimes\oo_S)
$$
obtained applying $\Hom (E_V,-)$ to \eqref{f}, we need to to prove that $\delta(f')=0$. By contradiction, assume $\delta(f')\in \Ext^1(E_V,V^\vee\otimes\oo_S)$ is the class of a nontrivial extension
\begin{equation}\label{extension}
0\to V^\vee\otimes\mathcal{O}_S\to E_1\to E_V\to 0
\end{equation}
that fits (by construction) in the following commutative diagram:
\begin{equation}\label{diagram}
 \xymatrix{
&&0&0\\
0\ar[r]&V^\vee\otimes\mathcal{O}_S\ar[r]&E_V\ar[u]\ar^f[r]&\oo_C(C-H)\ar[u]\ar[r]&0\\
0\ar[r]&V^\vee\otimes\mathcal{O}_S\ar@{=}[u]\ar[r]&E_1\ar[u]\ar^h[r]&E_V\ar^{f'}[u]\ar[r]&0.\\
&&K\ar^g[u]\ar@{=}[r]&K\ar[u]\\
&&0\ar[u]&0\ar[u]\\
}
\end{equation}
Since $\Ext^1(E_V,V^\vee\otimes\oo_S)\simeq H^1(E_V)^\vee\otimes V^\vee$, the element $\delta(f')$ also corresponds to a non-zero morphism from $H^1(E_V)$ to $V^\vee$. This implies that the extension \eqref{extension} fits in the following commutative diagram:
\begin{equation}\label{diag}
 \xymatrix{
&0&0\\
&\oo_S\ar[u]\ar@{=}[r]&\oo_S\ar[u]\\
0\ar[r]&V^\vee\otimes\mathcal{O}_S\ar[u]\ar[r]&E_1\ar[u]\ar^h[r]&E_V\ar[r]&0\\
0\ar[r]&H^1(E_V)\otimes\mathcal{O}_S\ar[u]\ar[r]&E_{C,\oo_C(H)}\ar[u]\ar[r]&E_V\ar@{=}[u]\ar[r]&0,\\
&0\ar[u]&0\ar[u]&0\ar[u]\\
}
\end{equation}
where the lowest row is \eqref{uni}. 

Since $H^1(E_{C,\oo_C(H)})=0$, the second column in the diagram splits, that is, $E_1\simeq \oo_S\oplus E_{C,\oo_C(H)}$. We will obtain a contradiction looking at the maps in diagram \eqref{diagram}. Since $h\circ g$ is injective and the $\oo_S$-factor of $E_1$ is contained  in the kernel of $h$, the image of $g:K\to E_1\simeq \oo_S\oplus E_{C,\oo_C(H)}$ is contained in $E_{C,\oo_C(H)}$ and its cokernel $E_V$ thus splits as $\oo_S\oplus \mathrm{Coker} (q\circ g)$, where $q: E_1\to E_{C,\oo_C(H)}$ is the obvious projection. This is a contradiction as $H^2(E_V)=0$. Therefore, $\delta(f')=0$ as required.

\end{proof}

\end{document}